\documentclass[11pt]{article}

\usepackage{amsmath,amsfonts,amssymb,amsthm} 

\usepackage{fullpage}
\usepackage[colorlinks,pdftex,pdftitle={A Hall-type condition for path covers in bipartite graphs}]{hyperref}
\usepackage[parfill]{parskip}

\newcommand{\ntwo}{{\mathsf{\Lambda}}}
\DeclareMathOperator{\deficiency}{{\mathsf\Lambda}-def}

\newcommand{\cycleA}{\mathcal C_{\mathsf{good}}}
\newcommand{\cycleB}{\mathcal C_{\mathsf{bad}}}
\newcommand{\cycleC}{\mathcal C_{\mathsf{ugly}}}

\newcommand{\defstyle}{\emph}

\newtheorem{theorem}{Theorem}
\newtheorem{lemma}[theorem]{Lemma}
\newtheorem{conjecture}[theorem]{Conjecture}
\newtheorem{proposition}[theorem]{Proposition}

\title{A Hall-type condition for path covers in bipartite graphs}
\author{Mikhail Lavrov \qquad Jennifer Vandenbussche \\
\small Department of Mathematics\\[-0.8ex]
\small Kennesaw State University\\[-0.8ex]
\small Kennesaw, GA\\
\small\tt \{mlavrov,jvandenb\}@kennesaw.edu}

\begin{document}

\maketitle

\begin{abstract}
Let $G$ be a bipartite graph with bipartition $(X,Y)$. Inspired by a hypergraph problem, we seek an upper bound on the number of disjoint paths needed to cover all the vertices of~$X$.  We conjecture that a Hall-type sufficient condition holds based on the maximum value of $|S|-|\ntwo(S)|$, where $S\subseteq X$ and $\ntwo(S)$ is the set of all vertices in $Y$ with at least two neighbors in $S$. This condition is also a necessary one for a hereditary version of the problem, where we delete vertices from $X$ and try to cover the remaining vertices by disjoint paths.  The conjecture holds when~$G$ is a forest, has maximum degree $3$, or is regular with high girth, and we prove those results in this paper.
\end{abstract}

\section{Introduction}

\subsection{Path covers of bipartite graphs}

Problems regarding path covers of graphs are ubiquitous in graph theory.  A \defstyle{path cover} of $G$ is a collection of vertex-disjoint paths in $G$ where the union of the vertices of the paths is $V(G)$.  Certainly the most well-studied example looks for a single path covering all vertices of $G$, i.e.\ a Hamiltonian path. Graphs with such a path are also called \defstyle{traceable}.  See~\cite{gould2003advances} for a survey of results in this area.  Determining whether a graph has a Hamiltonian path is NP-complete even for very restrictive classes of graphs; for example, Akiyama et al.~\cite{akiyama1980np} prove that it is NP-complete for 3-regular bipartite graphs.

In graphs that are not traceable, we may seek a path cover with as few paths as possible. For example, Magnant and Martin~\cite{Magnant2009pathcover} conjecture that a $d$-regular graph $G$ can be covered with at most $|V(G)|/(d+1)$ paths, and prove this when $d \leq 5$. Feige and Fuchs~\cite{feige20226regular} extend the result to~$d=6$.  In~\cite{Magnant2016pathcovermindegree}, Magnant et al.\ conjecture that a graph with maximum degree~$\Delta$ and minimum degree~$\delta$ needs at most $\max\left\{\frac{1}{\delta+1},\frac{\Delta-\delta}{\Delta+\delta}\right\}\cdot |V(G)| $ paths to cover its vertices, which they verify for $\delta \in \{1,2\}$ and which Kouider and Zamime~\cite{kouider2022preprint} prove for $\Delta \ge 2\delta$. For dense $d$-regular bipartite graphs, Han~\cite{han2018} proves that a collection of $|V(G)|/(2d)$ vertex-disjoint paths covers all but~$o(|V(G)|)$ vertices.

In this paper, we focus on a variant of the path cover problem for bipartite graphs: collections of vertex-disjoint paths that cover one partite set of the bipartite graph. Let an \defstyle{$(X,Y)$-bigraph} be a bipartite graph with a specified ordered bipartition $(X,Y)$. If $G$ is an $(X,Y)$-bigraph, a \defstyle{path~$X$-cover} of $G$ is a set of pairwise vertex-disjoint paths in $G$ that cover all of $X$.  

We seek a Hall-type condition for the existence of a path $X$-cover of $G$ with at most $k$ paths.  Let $S \subseteq X$, and let $\ntwo_G(S)$ be the set of all vertices in $Y$ that have at least two neighbors in $S$; in cases where there is only one graph $G$ under consideration, we will write $\ntwo_G(S)$ simply as $\ntwo(S)$. We define the \defstyle{$\ntwo$-deficiency of $S$} to be $\deficiency(G,S) := |S|-|\ntwo(S)|$, and the \defstyle{$\ntwo$-deficiency of $G$} to be
\[
	\deficiency(G) := \max\{\deficiency(G,S) : S \subseteq X\}.
\]
We conjecture the following: 

\begin{conjecture}\label{conjecture:path cover}
Every $(X,Y)$-bigraph $G$ has a path $X$-cover by at most $\deficiency(G)$ paths.
\end{conjecture}

If this conjecture holds, then for every $S \subseteq X$, there is a set of at most $\deficiency(G)$ vertex-disjoint paths whose intersection with $X$ is precisely $S$. To see this, just delete all the vertices in $X-S$ from $G$, which can only decrease the $\ntwo$-deficiency.

Conversely, suppose it is true that for every $S \subseteq X$, there is a set of at most $k$ vertex-disjoint paths whose intersection with $X$ is precisely $S$. Then for every $S$, these paths have at least $|S|-k$ internal vertices in $Y$ that are all elements of $\ntwo(S)$; therefore $|\ntwo(S)| \ge |S|-k$ for all $S$, which implies that $\deficiency(G) \ge k$. It follows that the condition in our conjecture is a \emph{necessary} one if we would like to draw the stronger conclusion in the preceding paragraph.

Our conjecture is a slightly weakened form of a conjecture on cycle covers proposed in~\cite{kostochka2021conditions}:

\begin{conjecture}\label{conjecture:super-cyclic}
Let $G$ be an $(X,Y)$-bigraph with the property that for all $S \subseteq X$ with $|S|>2$, $\deficiency(G,S) \le 0$. Then $G$ contains a cycle that covers all of $X$. 
\end{conjecture}

We claim Conjecture~\ref{conjecture:super-cyclic} implies Conjecture~\ref{conjecture:path cover}. Let $H$ be the graph obtained from $G$ by adding ${\deficiency(G)}$ more vertices to $Y$, each of which is adjacent to every vertex in $X$. Then for all $S \subseteq X$ with $|S|>2$ (and even with $|S|=2$), we have $\deficiency(H,S) \le 0$, since all the new vertices of $H$ are in $\ntwo_H(S)$. Now a cycle in $H$ covering all of $X$ yields a path $X$-cover of $G$ by at most $\deficiency(G)$ paths by deleting all the new vertices.

\subsection{Hypergraphs and the Gallai--Milgram theorem}

The setting of Conjecture~\ref{conjecture:path cover} can be translated into the language of hypergraphs and Berge paths in hypergraphs, and here we see the motivation for focusing on path cover of $X$.

Following the terminology of Berge~\cite{berge73}, a \defstyle{hypergraph} $H$ consists of a set of vertices $V(H)$ and a set of edges $E(H)$ where each edge $e \in E(H)$ is a subset of $V(H)$. (We allow edges of any size.) The \defstyle{subhypergraph of $H$ generated by a set $S \subseteq V(H)$} is the hypergraph with $V(H)=S$ and 
\[
	E(H) = \{e \cap S : e \in E(H), e \cap S \ne \varnothing\}.
\]
There are several notions of paths in hypergraphs that generalize paths in graphs. One such notion is that of a \defstyle{Berge path}: a sequence 
\[
	(v_0, e_1, v_1, e_2, v_2, \dots, e_\ell, v_\ell)
\]
where $v_0, v_1, \dots, v_\ell$ are distinct vertices in $V(H)$, $e_1, e_2, \dots, e_\ell$ are distinct edges in $E(H)$, and $\{v_{i-1}, v_i\} \subseteq e_i$ for all $i=1, \dots, \ell$.

Given a hypergraph $H$, we can define its \defstyle{incidence graph} to be the $(X,Y)$-bigraph $G$ with $X = V(H)$ and $Y = E(H)$ such that $xy \in E(G)$ if and only if $x \in X$, $y \in Y$, and $x \in y$. Berge paths in $H$ correspond to paths in $G$ that begin and end in $X$; these are vertex-disjoint in $G$ if and only if they are both vertex-disjoint and edge-disjoint in $H$.

If we define a \defstyle{Berge path cover} of the hypergraph $H$ to be a set of pairwise vertex- and edge-disjoint paths that cover all of $V(H)$, then Conjecture~\ref{conjecture:path cover} proposes a sufficient condition for $H$ to have a Berge path cover of size at most $k$. Moreover, the proposed sufficient condition is a necessary condition for every subhypergraph of $H$ to have a Berge path cover of size at most $k$.

This statement is reminiscent of the Gallai--Milgram theorem (\cite{gallai1960}, p.~298 in \cite{berge73}), which states that the vertices of any directed graph $D$ can be covered by at most $\alpha(D)$ disjoint paths, where $\alpha(D)$ is the independence number of $D$. (The weaker statement for undirected graphs clearly follows.)
For a hypergraph $H$, let a set $I \subseteq V(H)$ be \defstyle{strongly independent} (following the terminology of Berge) if $|e \cap I| \le 1$ for all $e \in E(H)$; let $\alpha(H)$, the \defstyle{strong independence number of $H$}, be the size of a largest strongly independent set in $H$. It would be natural to hope that $H$ has a path cover by at most~$\alpha(H)$ pairwise-disjoint paths. In \cite{muller1981oriented}, M\"uller proves such a generalization of the Gallai--Milgram theorem (and, in fact, a generalization of it to directed hypergraphs), but in a slightly different setting: M\"uller does not require the edges of a path to be distinct, and does not require the paths in the cover to be edge-disjoint, merely vertex-disjoint.

In our setting, the corresponding generalization is false. Translating from hypergraphs back into the language of graphs: a set $I \subseteq V(H)$ is strongly independent if and only if, in the incidence graph of $H$, $\ntwo(I) = \varnothing$. Generalizing to an arbitrary $(X,Y)$-bigraph $G$, let $S \subseteq X$ be \defstyle{$\ntwo$-independent} if $\ntwo(S) = \varnothing$, and let the \defstyle{$\ntwo$-independence number} $\alpha_\ntwo (G)$ be the size of a largest $\ntwo$-independent set. Note that if $S$ is $\ntwo$-independent, then $\deficiency(G,S)=|S|$, so $\alpha_\ntwo(G)$ is always at most $\deficiency(G)$. 

To see that an $(X,Y)$-bigraph $G$ may not have a path $X$-cover with at most $\alpha_\ntwo(G)$ paths, even if $G$ is balanced and has a high connectivity, consider the following family of examples. Fix an integer~$k$ between 1 and $n$, and let $X = \{x_1, \dots, x_n\}$ and $Y = \{y_1, \dots, y_n\}$ with  $x_i y_j \in E(G)$ when $i \le k$ or $j \le k$. Then $\alpha_\ntwo(G)=1$, since any two vertices share the neighbor $y_1$, but $\deficiency(G) = n-2k+1$ (choose $S=\{x_k,x_{k+1}, \dots, x_n\}$), and in fact it can be checked that a minimum path $X$-cover contains $n-2k+1$ paths.

However, in all the cases of Conjecture~\ref{conjecture:path cover} we consider where $G$ is a \emph{regular} graph, $\alpha_\ntwo(G)$ paths suffice for a path $X$-cover of $G$.  Whether this holds for all regular bigraphs~$G$ is an open question that would have far-reaching consequences. For example, a result of Singer~\cite{singer1938theorem} states that the incidence graph of any classical projective plane is Hamiltonian. The proof relies on algebra over finite fields, but the claim above would give a purely graph-theoretic reason that these incidence graphs are always traceable, since the incidence graph $G$ of any projective plane must have $\alpha_\ntwo(G)=1$.

More generally, a hypergraph $H$ is \defstyle{covering} if every pair of vertices of $H$ lie on a common edge: in other words, $\alpha(H)=1$. Lu and Wang~\cite{lu2021hamiltonian} prove that every $\{1,2,3\}$-uniform covering hypergraph has a Hamiltonian Berge cycle. This implies Conjecture~\ref{conjecture:path cover} for $(X,Y)$-bigraphs $G$ with maximum degree $3$ in $Y$ and $\alpha_\ntwo(G)=1$.

\subsection{Our results}

Our first result states that Conjecture~\ref{conjecture:path cover} holds for forests:

\begin{proposition}\label{result:forest}
If $G$ is an $(X,Y)$-bigraph with no cycles, then $G$ has a path $X$-cover of size at most $\deficiency(G)$.
\end{proposition}

To strengthen Proposition~\ref{result:forest}, we go in two directions: we consider graphs with low maximum degree and graphs with high girth. In the first case, we begin by proving:

\begin{theorem}\label{result:3-regular}
If $G$ is a $3$-regular $(X,Y)$-bigraph, then $G$ has a path $X$-cover of size at most~$\alpha_\ntwo(G)$.
\end{theorem}

The proof of Theorem~\ref{result:3-regular} begins by taking a $2$-factor of $G$, covering the graph (and, in particular,~$X$) with pairwise vertex-disjoint cycles. If we generalize to graphs with maximum degree $3$, we are unable to do this, but if we cover as much of $G$ with cycles as possible, we are left with a forest. Once we deal with the interaction between the forest and the cycles, we can combine the arguments of Theorem~\ref{result:3-regular} with Proposition~\ref{result:forest} to prove a result for all graphs with maximum degree $3$:

\begin{theorem}\label{result:max degree 3}
If $G$ is an $(X,Y)$-bigraph with maximum degree at most $3$, then $G$ has a path $X$-cover of size at most~$\deficiency(G)$.
\end{theorem}

It is particularly interesting to strengthen Theorem~\ref{result:3-regular} to Theorem~\ref{result:max degree 3} because if $G$ has maximum degree at most $3$, then so does every subgraph of $G$. As a result, we obtain a necessary and sufficient condition for an $(X,Y)$-bigraph $G$ of maximum degree $3$ to have the property that for all $S \subseteq X$, there is a set of at most $\deficiency(G)$ pairwise vertex-disjoint paths whose intersection with $X$ is precisely $S$.

Conjecture~\ref{conjecture:path cover} holds for regular bigraphs of any degree if we add a condition on the \defstyle{girth} of $G$, that is, the length of the shortest cycle in $G$.

\begin{theorem}\label{result:high girth}
Let $G$ be an $(X,Y)$-bigraph with maximum degree at most $d$ and girth at least $4ed^2+1$, and assume that there exists a collection of pairwise vertex-disjoint cycles in $G$ that cover all of~$X$. (In particular, such a collection is guaranteed to exist if $G$ is $d$-regular.)

Then $G$ has a path $X$-cover of size at most $\alpha_\ntwo(G)$. 
\end{theorem}
\section{Forests}

\begin{proof}[Proof of Proposition~\ref{result:forest}]
We may assume that $G$ has no leaves in $Y$, since a vertex in $Y$ of degree $1$ does not contribute to $\deficiency(G,S)$ for any $S$, and it does not help cover more of $X$ by paths. We may also assume that $G$ is a tree; if $G$ has multiple components, we can solve the problem on each component separately.

We induct on $|X|$. When $|X|=1$, we have $\deficiency(G) = \deficiency(G, X) = 1$, and we can cover $X$ by a single path of length $0$.

When $|X|>1$, consider $G$ as a rooted tree with an arbitrary root in $X$. Let $x \in X$ be a leaf of $G$ at the furthest distance possible from the root, and let $y \in Y$ be the parent vertex of $x$.

\textbf{Case 1:} $y$ has other children.\nopagebreak

Let $x_1, \dots, x_k$ be all the children of $y$ (including $x$); by the case, $k\ge 2$. Since $x$ was chosen to be as far from the root as possible, each $x_i$ must be a leaf. Delete $x_1, x_2, \dots, x_k, y$ from $G$ to get $G'$.

Let $S \subseteq X - \{x_1, \dots, x_k\}$ be the set such that $\deficiency(G') = \deficiency(G', S)$. We claim that
\[
\deficiency(G, S \cup \{x_1, \dots, x_k\}) = \deficiency(G', S) + k - 1.
\] 
On one hand, $|S \cup \{x_1, \dots, x_k\}| = |S|+k$. On the other hand, $\ntwo_G(S \cup \{x_1, \dots, x_k\}) = \ntwo_{G'}(S) \cup \{y\}$, so $|\ntwo_G(S \cup \{x_1, \dots, x_k\})| = |\ntwo_{G'}(S)|+1$. In particular, $\deficiency(G) \ge \deficiency(G',S) + k-1$.

By the inductive hypothesis, $G'$ has a path $X$-cover by at most $\deficiency(G',S)$ paths. Add $k-1$ more paths to that set: the path $(x_1, y, x_2)$ and the length-$0$ paths $(x_3), \dots, (x_k)$. This is a path $X$-cover of $G$ by at most $\deficiency(G',S) + k-1 \le \deficiency(G)$ paths, completing the case.

\textbf{Case 2:} $y$ has no other children.

Let $x^* \in X$ be the parent vertex of $y$.

\textbf{Case 2a:} $\deg(x^*) \le 2$ (this includes the case where $x^*$ is the root and $\deg(x^*)=1$).

Delete $x$ and $y$ from $G$ to get $G'$. For every $S \subseteq X - \{x\}$, we have $\deficiency(G,S) = \deficiency(G',S)$, since~$y$ cannot be in $\ntwo_G(S)$ and all other vertices of $Y$ are still in $G'$. Therefore $\deficiency(G) \ge \deficiency(G')$. 

By the inductive hypothesis, $G'$ has a path $X$-cover by at most $\deficiency(G')$ paths. By the case, $x^*$ is a leaf of $G'$ (or an isolated vertex), so the path that covers $x^*$ must begin or end at $x^*$. Extend that path to go through $y$ and $x$, and we get a path $X$-cover of $G$ by $\deficiency(G') \le \deficiency(G)$ paths, completing the case.

\textbf{Case 2b:} $\deg(x^*) \ge 3$.

Let $y_1, \dots, y_k$ be all of the children of $x^*$ (including $y$); by the case, $k\ge 2$. No vertices of $y$ are leaves, so each has a child. By our choice of $x$, those children are all as far from the root as possible, so they must all be leaves. If any of $y_1, \dots, y_k$ have multiple children, then we can proceed as in \textbf{Case~1}, so assume each $y_i$ has a single child $x_i$. Delete $x_k$ and $y_k$ from $G$ to get $G'$.

Let $S \subseteq X - \{x_k\}$ be the set such that $\deficiency(G') = \deficiency(G',S)$. We may assume that $x^* \notin S$ by one of the following modifications:
\begin{itemize}
\item If $x^* \in S$ and $x_1 \notin S$, replace $S$ by $S' = S \cup \{x_1\} - \{x^*\}$. Then $|S'| = |S|$ and $|\ntwo_{G'}(S')| \le |\ntwo_{G'}(S)|$: $y_1$ is in neither $\ntwo_{G'}(S)$ nor $\ntwo_{G'}(S')$, and no other vertices in $Y$ have any neighbors in $S'$ that they did not have in $S$. So $\deficiency(G',S') \ge \deficiency(G',S)$.

\item If $x^* \in S$ and $x_1 \in S$, replace $S$ by $S' := S-\{x^*\}$. Then $|S'| = |S|-1$, but $|\ntwo_{G'}(S')| \le |\ntwo_{G'}(S)|-1$ as well, since $y_1 \in \ntwo_{G'}(S)$ but $y_1 \notin \ntwo_{G'}(S')$. (No other vertices in $Y$ have any neighbors in $S'$ that they did not have in $S$.) So $\deficiency(G',S') \ge \deficiency(G',S)$.
\end{itemize}
When $x^* \notin S$, we have $\deficiency(G, S \cup \{x_k\}) = \deficiency(G', S)+1$, because $|S \cup \{x_k\}| = |S|+1$, while~$\ntwo_{G}(S \cup \{x_k\})=\ntwo_{G'}(S)$. Therefore $\deficiency(G) \ge \deficiency(G') + 1$.

By the inductive hypothesis, $G'$ has a path $X$-cover by at most $\deficiency(G')$ paths. Add the path~$(x_k)$ to get a path $X$-cover of $G$ by $\deficiency(G') + 1 \le \deficiency(G)$ paths, completing the case and the proof.
\end{proof}

\section{3-regular graphs}

It is a standard result (Corollary 3.1.13 in~\cite{west1996introduction}) that every regular bipartite graph has a perfect matching. Removing a perfect matching from a $d$-regular bipartite graph leaves a $(d-1)$-regular bipartite graph, which also has a perfect matching. The union of the two matchings provides a cover of $G$ by vertex-disjoint cycles, giving the following lemma (which is also well-known):

\begin{lemma}
\label{lemma:cycle cover}
If $G$ is a regular bipartite graph, then $G$ has a cycle cover.
\end{lemma}

The existence of this lemma is the primary reason that this proof is simpler than the proof of Theorem~\ref{result:max degree 3} in the next section.  That proof begins with the same ideas, but must deal with vertices of $X$ that are not part of the initially chosen collection of cycles.

\begin{proof}[Proof of Theorem~\ref{result:3-regular}]
By Lemma~\ref{lemma:cycle cover}, we can take a cycle cover $\mathcal C$ of $G$. Let $S$ be any maximal $\ntwo$-independent subset of $X$ such that each cycle in $\mathcal C$ contains at most one vertex of $S$. To prove the claim, it suffices to construct a path cover of $G$ with exactly $|S|$ paths. We give an algorithm for this below.

Let $H$ be a subgraph of $G$ that will change over the course of the algorithm; initially, $H$ will consist of the $|S|$ cycles in $\mathcal C$ containing a vertex of $S$. We will extend $H$ to a spanning subgraph of $G$, while maintaining the properties (1) $H$ has $|S|$ components, and (2) each component of $H$ is traceable.

At each step of the algorithm, choose a cycle $C \in \mathcal C$ that is not yet contained in $H$, and $x(C) \in V(C) \cap X$.  In most cases, we make this choice arbitrarily. Occasionally, we will want to make sure that a particular vertex $w$ on a cycle $C$ not yet in $H$ will never become $x(C)$. To do so, we select~$C$ to be processed next, and choose an arbitrary vertex in $V(C) \cap X$ other than $w$ to be $x(C)$. To indicate that we do this, we say that we \defstyle{Do Something Else} with $w$; we provide details about this choice later.  

Suppose we have selected $C$ and $x(C)$.  By the maximality of $S$, we have $\ntwo(S \cup \{x(C)\}) \ne \varnothing$, so there is some vertex $s \in S$ such that $x(C)$ and $s$ have a common neighbor $y$. Let $C(s)$ be the cycle in $\mathcal C$ containing $s$. The vertex $y$ must lie on either $C$ or $C(s)$, since otherwise $y$ would have four neighbors: $x(C)$, $s$, and its two neighbors on the cycle in $\mathcal C$ containing $y$. We extend $H$ by adding cycle $C$ to $H$, and either the edge $x(C)y$ (if $y$ lies on $C(s)$) or $sy$ (if $y$ lies on $C$). This ends one step of the algorithm.

This step maintains the property that $H$ has $|S|$ components, since cycle $C$ has been joined to an existing component of $H$. To maintain the property that each component of $H$ is traceable, we must clarify when we Do Something Else. 

Consider an arbitrary $s \in S$; let $C(s)$ be the cycle of $\mathcal C$ containing $s$, and let $y_1, y_2$ be the two neighbors of $s$ along $C(s)$. Initially, the component of $H$ containing $s$ is just $C(s)$.  There are three ways that $C(s)$ can potentially be added to $H$, namely via an edge from any of $s$, $y_1$, or $y_2$ going to another cycle in $\mathcal C$. The component remains traceable if any one of these edges is used to extend it: in that case, we can extend that edge to a Hamiltonian path by going the long way around both cycles. The component also remains traceable if it is extended both using an edge from $s$ and using an edge from $y_1$. In that case, delete edge $sy_1$, obtaining a long path containing $s$ and $y_1$ joining two cycles; extend that path by going the long way around both of those cycles. The same is true if $y_1$ is replaced by $y_2$.

However, we must ensure that the component of $H$ containing $s$ is never extended by using edges from both $y_1$ and $y_2$. Suppose that a step of the algorithm extended the component of $H$ containing~$s$ via an external edge to $y_1$, and $y_2$ has a neighbor $w$ in some $C \in \mathcal C$ not yet contained in $H$.  In this situation, we Do Something Else with $w$.  This ensures the component of $H$ containing $s$ cannot be extended using edges from both $y_1$ and $y_2$, because one of those edges goes to $C$, and $C$ will become part of $H$ in the next step of the algorithm and hence will not be considered at later stages of the algorithm. As a result, no component of $H$ is ever prevented from being traceable.

At the end of the algorithm, we have a spanning subgraph $H$ with $|S|$ traceable components. By taking a Hamiltonian path in each component, we obtain a path cover of $G$ with $|S|$ paths, completing the proof.
\end{proof}

\section{Graphs with maximum degree 3}

\begin{proof}[Proof of Theorem~\ref{result:max degree 3}]
We will prove the theorem by describing an algorithm that constructs a path $X$-cover $\mathcal P$ and a set $S\subseteq X$ with $|\mathcal P| = \deficiency(G, S)$.

To begin the algorithm, let $\mathcal C$ be a collection of vertex-disjoint cycles in $G$ satisfying the following conditions:
\begin{enumerate}
\item The union of the cycles contains as many vertices of $G$ as possible.
\item Subject to condition 1, there are as few cycles as possible.
\end{enumerate}
As a consequence of condition 1, deleting the vertices in $\mathcal C$ from $G$ leaves a forest, which we call $F$.

In the next phase of the algorithm, we process the cycles in $\mathcal C$, one at a time. This phase has two goals. First, for each $C \in \mathcal C$, we will choose a designated vertex $x(C) \in V(C) \cap X$. Intuitively, $x(C)$ will be the only vertex of $C$ which \emph{may} become part of the high-$\ntwo$-deficiency set $S$ we construct. We define $y^+(C)$ and $y^-(C)$ to be the two neighbors of $x(C)$ along $C$. Second, we will split $\mathcal C$ into three sets: $\cycleA$, $\cycleB$, and $\cycleC$. Intuitively, if $C \in \cycleA$, then $x(C)$ is far from any problems; if $C \in \cycleB$, then $x(C)$ is too close to the forest $F$; finally, if $C \in \cycleC$, then $x(C)$ is too close to $x(D)$ for some $D \in \cycleA \cup \cycleB$.

In most cases, we arbitrarily choose an unprocessed cycle $C$ to process next, and arbitrarily choose $x(C) \in V(C) \cap X$. Occasionally, as in the proof of Theorem~\ref{result:3-regular}, we will want to make sure that a particular vertex $w$ on an unprocessed cycle $C$ will never become $x(C)$. To do so, we select $C$ to be processed next, and choose an arbitrary vertex in $V(C) \cap X$ other than $w$ to be $x(C)$. As in Theorem~\ref{result:3-regular}, to indicate that we do this, we say that we \defstyle{Do Something Else} with $w$.

To decide what to do with a cycle $C$ as we process it, we consider the following cases, \emph{in order}, choosing the first that applies:

\textbf{Case 1:} $x(C)$ has a common neighbor with $x(D)$ for some $D \in \cycleA \cup \cycleB$, and that common neighbor lies on either $C$ or $D$.  In other words, at least one edge
\[
	e(C) \in \{x(C)y^+(D), x(C)y^-(D), y^+(C)x(D), y^-(C) x(D)\}
\]
must exist in $G$. (If multiple choices of $D$ or of $e(C)$ are possible, then fix one of them.)
In this case, we place $C$ in $\cycleC$; we say that \defstyle{$C$ attaches to $D$ at $u$}, where $u$ is the endpoint of $e(C)$ in~$D$. We save the edge $e(C)$ for reference; later, we will use it to extend a path covering $D$ to also cover~$C$. 

Additionally, if $e(C) = x(C) y^{\pm}(D)$ and the vertex $y^{\mp}(D)$ (that is, whichever of $y^+(D), y^-(D)$ is not an endpoint of $e(C)$) is adjacent to a vertex $w$ on an unprocessed cycle, we Do Something Else with~$w$.

\textbf{Case 2:} At least one of $x(C)$, $y^+(C)$, or $y^-(C)$ has a neighbor in $F$.  In this case, we place $C$ in~$\cycleB$. Additionally, if $y^{\pm}(C)$ has a neighbor in $F$ and $y^{\mp}(C)$ has a neighbor $w$ on an unprocessed cycle, we Do Something Else with $w$.

\textbf{Case 3:} Neither case 1 nor case 2 occurs. In this case, we simply place $C$ in $\cycleA$.

This concludes the second phase (or the \defstyle{processing phase}) of the algorithm.

In the third phase of the algorithm, we create an auxiliary graph $F^*$ (which is not precisely a subgraph of $G$) containing $F$ and some extra vertices representing the elements of $\cycleB$. For each $C \in \cycleB$:
\begin{itemize}
\item We add $x(C)$ to $F^*$, together with the edge to its neighbor in $F$, if there is one.
\item We add an artificial vertex $y^*(C)$ to $F^*$ that is adjacent to $x(C)$ and to the neighbors of both $y^+(C)$ and $y^-(C)$ in $F$, if these exist.
\end{itemize}
Before we continue, we must show that $F^*$ is a forest. Suppose for the sake of contradiction that~$F^*$ contains a cycle.  Since $F$ is acyclic, this cycle must contain either $x(C)$ or $y^*(C)$ for at least one $C \in \cycleB$.

First, consider the case that the cycle only includes the vertex $y^*(C)$ for a single $C \in \cycleB$. This means that there is a path $P$ from $y^+(C)$ to $y^-(C)$ of length at least $3$, whose internal vertices are in $F$. Now we can modify $C$, replacing $x(C)$ and the edges $y^+(C)x(C), x(C)y^-(C)$ by $P$. The resulting cycle contains more vertices that $C$, violating condition 1 in the definition of $\mathcal C$. 

Similarly, if the cycle in $F^*$ includes only the vertices $x(C)$ and $y^*(C)$ for a single $C \in \cycleB$, we can expand $C$ to include some vertices in $F$. This also violates condition 1 in the definition of $\mathcal C$. 

Finally, consider the case that the cycle in $F^*$ includes vertices $x(C)$ and/or $y^*(C)$ for multiple $C \in \cycleB$. In this case, we can extend it to a cycle in $G$: every time the cycle in $F^*$ visits $y^*(C)$, we can replace that visit by a path that enters $C$ via $y^{\pm}(C)$, goes around $C$, and leaves via either~$y^{\mp}(C)$ or $x(C)$. This cycle in $G$ contains at least as many vertices as the cycles from $\mathcal C$ it uses: it misses at most the vertex $x(C)$ from each of them, but includes a vertex in $F$ between any two of the cycles in $\mathcal C$. Therefore, we can replace multiple cycles in $\mathcal C$ by a single cycle through at least as many vertices, violating condition 2 in the definition of $\mathcal C$.

In all cases, we arrive at a contradiction, so we can conclude that $F^*$ is a forest. We give it the structure of an $(X',Y')$-bigraph by defining:
\begin{align*}
	X' &= (X \cap V(F)) \cup \{x(C) : C \in \cycleB\}, \\
	Y' &= (Y \cap V(F)) \cup \{y^*(C) : C \in \cycleB\}.
\end{align*}
By Proposition~\ref{result:forest}, we can find a path $X'$-cover $\mathcal P'$ of $F^*$ and a set $S' \subseteq X'$ such that $\deficiency(F^*,S') \ge |\mathcal P'|$. We may assume that both endpoints of every path in $\mathcal P'$ are in $X'$, not $Y'$.

Because $X' \subseteq X$, we have $S' \subseteq X$ as well; moreover, none of the vertices $x(C)$ for $C \in \cycleB$ have common neighbors outside $F$, or else one of the cycles would have been handled by Case 1 of the processing phase instead. Therefore $\deficiency(G,S') = \deficiency(F^*,S')$. 

We are now ready to construct the path $X$-cover $\mathcal P$ in $G$ and a set $S \subseteq X$ with $|\mathcal P| \leq \deficiency(G,S)$. Let
$S = S' \cup \{x(C) : C \in \cycleA\}$.
The vertices $\{x(C) : C \in \cycleA\}$ have no common neighbors with each other or with any vertex in $S'$.  This is ensured by Case 1 and Case 2 of the processing phase, where any cycle $C$ for which $x(C)$ did have such a common neighbor would be placed in~$\cycleB$ or~$\cycleC$ instead. Therefore $\deficiency(G,S) = \deficiency(G, S') + |\cycleA|$. We transform $\mathcal P'$ into $\mathcal P$ with~$|\cycleA|$ additional paths.

For each $C \in \cycleA$, we add a single path to $\mathcal P$ that covers the vertices of $X$ on both $C$ and any cycles in $\cycleC$ that attach to $C$. There are three possibilities:
\begin{itemize}
\item If there are no cycles in $\cycleC$ that attach to $C$, we take a path that goes around $C$.

\item If there is one cycle $C_1 \in \cycleC$ that attaches to $C$, we obtain a path that covers both~$C$ and~$C_1$ by taking both cycles, adding edge $e(C_1)$, and deleting an edge incident on~$e(C_1)$ from both cycles.

\item If there are two cycles $C_1, C_2 \in \cycleC$ that attach to $C$, then $e(C_1)$ and $e(C_2)$ cannot both have endpoints in $\{y^+(C), y^-(C)\}$, because once one such cycle is added to $\cycleC$ in Case 1, we Do Something Else to prevent a second cycle of this type from appearing. (This is also why three cycles in $\cycleC$ cannot attach to $C$.) Without loss of generality, $C_1$ and $C_2$ attach to $C$ at $x(C)$ and $y^+(C)$.

We obtain a path that covers $C$, $C_1$, and $C_2$ by taking all three cycles, adding edges $e(C_1)$ and $e(C_2)$, deleting edge $x(C) y^+(C)$ from $C$, and deleting an edge from each of $C_1$ and $C_2$ incident on~$e(C_1)$ and~$e(C_2)$, respectively.
\end{itemize}

Finally, we must transform the paths in $\mathcal P'$ that use the vertices $x(C)$ or $y^*(C)$ for some $C \in \cycleB$ into paths in $G$ that cover the cycles in $\cycleB$, as well as any cycles in $\cycleC$ that attach to them. 

For each $C \in \cycleB$, there is a path $P_x \in \mathcal P'$ that covers $x(C)$, and possibly a different path $P_y \in \mathcal P'$ that passes through $y^*(C)$. Without changing these paths outside $\{x(C), y^*(C)\}$, we modify them to cover $C$ and any cycles in $\cycleC$ that attach to $C$. During this process, the host graph of the paths in $\mathcal P'$ is unclear, but once we have considered all of $\cycleB$, the paths will all be paths in $G$.

There are multiple possibilities for how the modification is done:
\begin{itemize}
\item $P_x$ contains $x(C)$ but not $y^*(C)$, and $P_y$ does not exist. Then $x(C)$ must be an endpoint of~$P_x$, since $x(C)$ has only one neighbor in $F^*$ other than $y^*(C)$.

We extend $P_x$ to go around the cycle $C$, ending at either $y^+(C)$ or $y^-(C)$. If there is a cycle $C_1 \in \cycleC$ that attaches to $C$ at $y^{\pm}(C)$, we choose $P_x$ to go around $C$ so that it ends at the endpoint of $e(C_1)$; then, extend $P_x$ to use $e(C_1)$ and go around $C_1$. As before, there cannot be two cycles $C_1, C_2 \in \cycleC$ that attach to $C$ at $y^+(C)$ and $y^-(C)$, since we Do Something Else to prevent this.

\item $P_x$ goes from $x(C)$ to $y^*(C)$ and continues to $F$, and $P_y$ does not exist.

The edge used from $y^*(C)$ comes from an edge in $G$ from $y^{\pm}(C)$ to $F$; we modify $P_x$ to go around the cycle $C$ from $x(C)$ to $y^{\pm}(C)$. If there is a cycle $C_1 \in \cycleC$ that attaches to $C$ at~$x(C)$, then $x(C)$ cannot have a neighbor in $F$, so it is an endpoint of $P_x$. We extend~$x(C)$ in the other direction, prepending a path that goes around $C_1$ and takes edge $e(C_1)$ to~$x(C)$. There cannot be a cycle $C_2 \in \cycleC$ that attaches to $C$ at $y^{\mp}(C)$, again because we Do Something Else to prevent it.

\item $P_x$ contains $x(C)$, and a different path $P_y$ contains $y^*(C)$.

In this case, $x(C)$ must be an endpoint of $P_x$. We leave $P_x$ unchanged, unless there is a cycle $C_1 \in \cycleC$ that attaches to $C$ at $x(C)$. In this case, $x(C)$ has no neighbors in $F$, so $P_x$ must be a path of length $0$. We replace $P_x$ by a path that covers $C_1$ and ends with $e(C_1)$, also covering $x(C)$.

Meanwhile, $P_y$ must enter and leave $y^*(C)$ by edges other than $x(C)y^*(C)$; in $G$, these correspond to two edges between $\{y^+(C), y^-(C)\}$ and $F$. We modify $P_y$, replacing $y^*(C)$ by the $y^+(C),y^-(C)$-path that goes around $C$, covering all its vertices except the previously covered $x(C)$. 
\end{itemize}
Once this process is complete, we add these modified paths to $\mathcal P$, along with the paths that covered the cycles in $\cycleA$. The collection $\mathcal P$ is a path $X$-cover; just as $\mathcal P'$ did, it still covers all vertices of $X'$, but now it also covers all vertices of the cycles in $\mathcal C$. This completes the proof, since $|\mathcal P| = |\mathcal P'| + |\cycleA| \le \deficiency(G,S') + |\cycleA| = \deficiency(G,S)$.
\end{proof}

\section{Graphs with high girth}

We will use the Lov\'asz Local Lemma~\cite{erdos1975problems} to prove Theorem~\ref{result:high girth}, in the form stated below.

\begin{lemma}[The Local Lemma; Lemma~5.1.1 in~\cite{alon08}]\label{local-lemma}
Let $A_1, A_2, \dots, A_N$ be events in an arbitrary probability space. A directed graph $D = (V,E)$ on the set of vertices $V = \{1,2,\dots,N\}$ is called a \defstyle{dependency digraph} for the events $A_1, \dots, A_N$ if for each $i$, $1 \le i \le N$, the event $A_i$ is mutually independent of all the events $\{A_j : (i,j) \notin E\}$. Suppose that $D = (V,E)$ is a dependency digraph for the above events and suppose there are real numbers $x_1, \dots, x_N$ such that $0 \le x_i < 1$ and $\Pr[A_i] \le x_i \prod_{(i,j) \in E} (1-x_j)$ for all $1 \le i \le N$. Then \[
	\Pr\left[ \bigwedge_{i=1}^N \overline {A_i}\right] \ge \prod_{i=1}^N (1-x_i).
\]
In particular, with positive probability, no event $A_i$ holds.
\end{lemma}

A symmetric version of Lemma~\ref{local-lemma} is often used, where $\Pr[A_i] = p$ for all $i$; by setting $x_i = e \cdot \Pr[A_i]$ for all $i$, and using the inequality $(1 - \frac1{x+1})^x \ge \frac1e$, valid for all $x \ge0$, the hypotheses of the lemma are satisfied. In our case, the probabilities of our events $A_i$ will vary, but we will pursue mostly the same strategy. We will still set $x_i = e \cdot \Pr[A_i]$ for all $i$; because event $A_i$ will depend on two types of other events, we will use the inequality $(1 - \frac1{2x+1})^x \ge e^{-1/2}$ (also valid for all $x \ge 0$) on two parts of the product, instead.

\begin{proof}[Proof of Theorem~\ref{result:high girth}]
Let $G$ be an $(X,Y)$-bigraph with maximum degree at most $d$ and girth $g \ge 4ed^2 + 1$. Additionally, let $\mathcal C$ be a collection of pairwise vertex-disjoint cycles in $G$ that cover $X$. This collection exists either by assumption or (if $G$ is taken to be $d$-regular) by Lemma~\ref{lemma:cycle cover}. The girth condition on $G$ guarantees that the cycles in $\mathcal C$ are relatively long, so there cannot be too many of them; in fact, we will show that there cannot be more than $\alpha_\ntwo(G)$ of them.

If $v \in X$, we will write $C(v)$ for the cycle in $\mathcal C$ containing $v$.  Furthermore, if $v_1, v_2 \in X$, we will write $v_1 \sim v_2$ to mean that $C(v_1) \ne C(v_2)$ and $v_1$ and $v_2$ have a common neighbor in $Y$.

We choose a set $S$ randomly, by selecting one vertex of $X$ uniformly at random from each cycle in $\mathcal C$. Let $\{u_1, v_1\}, \dots, \{u_N, v_N\}$ be an enumeration of the (unordered) pairs of vertices in $X$ such that $u_i \sim v_i$. For each $i$, $1 \le i \le N$, we let $A_i$ be the event that $u_i \in S$ and $v_i \in S$. As a result, the conjunction $\overline{A_1} \land \dots \land \overline{A_n}$ is exactly the claim that $S$ is $\ntwo$-independent. If we can satisfy the hypotheses of Lemma~\ref{local-lemma} for the events $A_1, \dots, A_N$, then this conjunction occurs with positive probability, and therefore $|\mathcal C| = |S| \le \alpha_\ntwo(G)$, proving the theorem.

We define the dependency digraph $D$ to include an edge $(i,j)$ whenever the four vertices~$u_i$,~$v_i$,~$u_j$, and~$v_j$ do \emph{not} lie on four distinct cycles. If $C(u_i)$ has length $2\ell_1$ and $C(v_i)$ has length $2\ell_2$, we define $x_i = \frac{e}{\ell_1 \ell_2}$; for reference, $\Pr[A_i] = \frac1{\ell_1 \ell_2}$.
 
We will show that the hypotheses of Lemma~\ref{local-lemma} hold with the choices made above. To do this, we must put a lower bound on
\[
	x_i \prod_{(i,j) \in E(D)} (1 - x_j)
\]
for an arbitrary event $A_i$.

This product consists of two types of events $A_j$. The first type consists of those $A_j$ for which either~$u_j$ or~$v_j$ lies on $C(u_i)$ (including $u_j=u_i$ or $v_j=u_i$). There are at most $\ell_1 d^2$ events $A_j$ of this type. For each of them, one of $C(u_j), C(v_j)$ is the same as $C(u_i)$ and has length $\ell_1$, and the other has length at least $g$. Therefore $1 - x_j \ge 1 - \frac{2e}{\ell_1 g}$, for an overall product of at least $(1 - \frac{2e}{\ell_1 g})^{\ell_1 d^2}$. 

The second type of events $A_j$ such that $(i,j) \in E(D)$ consists of those $A_j$ for which either $u_j$ or $v_j$ lies on $C(v_i)$. By a similar argument, there are at most $\ell_2 d^2$ events $A_j$ of this type, and for each of them, $1 - x_j \ge 1 - \frac{2e}{\ell_2 g}$, so
\[
	x_i \prod_{(i,j) \in E(D)} (1 - x_j) \ge \frac e{\ell_1 \ell_2} \left(1 - \frac {2e}{\ell_1 g}\right)^{\ell_1 d^2} \left(1 - \frac {2e}{\ell_2 g}\right)^{\ell_2 d^2}.
\]
Recall that $g \ge 4ed^2+1$; a lower bound on $\frac{\ell_1 g}{2e} \geq 2\ell_1 d^2 + \frac{\ell_1}{2e}$ is $2\ell_1 d^2 + 1$. Applying $(1 - \frac1{2x+1})^x \ge e^{-1/2}$, we get
\[
	\left(1 - \frac {2e}{\ell_1 g}\right)^{\ell_1 d^2} \ge \left(1 - \frac1{2\ell_1 d^2+1}\right)^{\ell_1 d^2} \ge e^{-1/2}
\]
and similarly $\left(1 - \frac {2e}{\ell_2 g}\right)^{\ell_2 d^2} \ge e^{-1/2}$. Therefore
\[
	\frac e{\ell_1 \ell_2} \left(1 - \frac {2e}{\ell_1 g}\right)^{\ell_1 d^2} \left(1 - \frac {2e}{\ell_2 g}\right)^{\ell_2 d^2} \ge \frac e{\ell_1 \ell_2} \cdot e^{-1/2} \cdot e^{-1/2} = \frac1{\ell_1 \ell_2} = \Pr[A_i]
\]
and the conditions of Lemma~\ref{local-lemma} are satisfied. 

We conclude that with positive probability, $S$ is $\ntwo$-independent, and therefore $|\mathcal C| = |\mathcal S| \le \alpha_\ntwo(G)$. The theorem follows, since we can find a path $X$-cover of $G$ of size $|\mathcal C|$ by removing a vertex of $Y$ from each cycle in $\mathcal C$.
\end{proof}

\bibliographystyle{plain}

\end{document}